\documentclass[12pt,reqno]{article}

\usepackage[usenames]{xcolor}
\usepackage{amssymb}
\usepackage{amsmath}
\usepackage{amsthm}
\usepackage{amsfonts}
\usepackage{array}

\usepackage[colorlinks=true,
linkcolor=webgreen,
filecolor=webbrown,
citecolor=webgreen]{hyperref}

\definecolor{webgreen}{rgb}{0,.5,0}
\definecolor{webbrown}{rgb}{.6,0,0}

\usepackage{color}
\usepackage{fullpage}
\usepackage{longtable} 
\usepackage{graphics}
\usepackage{breakurl}

\newcommand{\seqnum}[1]{\href{https://oeis.org/#1}{\rm \underline{#1}}}

\theoremstyle{plain}
\newtheorem{theorem}{Theorem}
\newtheorem{corollary}[theorem]{Corollary}
\newtheorem{lemma}[theorem]{Lemma}
\newtheorem{proposition}[theorem]{Proposition}

\theoremstyle{definition}

\newtheorem{question}[theorem]{Question}

\begingroup
\catcode`\,=\active
\gdef,{\mathchar"613B\allowbreak}   
\endgroup
\newcommand{\breakablemathcommas}{\mathcode`\,="8000\relax}

\begin{document}

\begin{center}
\vskip 1cm{\LARGE\bf Strongly Pseudoperfect Numbers}
\vskip 1cm
\large
Audrey Wang\\
Hopkins School\\
New Haven, CT 06515\\
USA \\
\href{mailto:awang28@students.hopkins.edu}{\tt awang28@students.hopkins.edu}\\ \ \\
Joshua Zelinsky\\
Department of Mathematics\\ 
Hopkins School\\ 
New Haven, CT 06515\\
USA\\
\href{mailto:jzelinsky@hopkins.edu}{\tt jzelinsky@hopkins.edu} \\
\end{center}

\vskip .2 in

\begin{abstract} A positive integer $n$ is said to be strongly pseudoperfect if there is a subset $S$ of the positive divisors of $n$ such that $d \in S$ if and only if $n/d \in S$, and $\sum_{d \in S} d = 2n$. We construct new infinite families of strongly pseudoperfect numbers, and also prove new restrictions on when a positive integer $n$ can be strongly pseudoperfect.
    
\end{abstract}

\section{Introduction}
\label{sec:intro}

Let $\sigma(n)$ be the sum of the divisors of $n$. A number $n$ is said to be perfect if $\sigma(n)=2n$. The first few perfect numbers are $6$, $28$, $496$, and $8128 $ and are listed as \seqnum{A000396} in the OEIS. Similarly, a number $n$ is said to be deficient if $\sigma(n) < 2n$ and abundant if $\sigma(n) > 2n$. We write $h(n)=\sigma(n)/n$. Note that $h(n)$ is frequently referred to as the abundancy of $n$. We also write $\tau(n)$ for the number of positive divisors of $n$.

Two of the oldest unsolved problems are whether there are infinitely many even perfect numbers and whether there are any odd perfect numbers. We have the following classical result, with one direction due to Euclid, and the other to Euler. 

\begin{theorem}[Euclid, Euler]\label{Euclid-Euler} If $n$ is an even number, then $n$ is perfect if and only if $n=2^{p-1}(2^p-1)$ where $2^p-1$ is prime. 
\end{theorem}

Primes of the form $2^p-1$ are called Mersenne primes. It is not hard to show that if $2^p-1$ is prime, then so is $p$ (hence the choice of letter), but the other direction does not follow. For example, $2^{11} -1= 23 \cdot 89$.

Euler also proved the following statement about odd perfect numbers.

\begin{theorem} \label{Euler's theorem for odd perfects} If $n$ is an odd perfect number, then $n= p^e m^2$ where $p$ is prime, $p \equiv e \equiv 1$ (mod $4$), and $\gcd(p,m)=1$.
\end{theorem}

Euler's theorem for odd perfect numbers is in fact very weak. It is not hard to modify the proof to show that the result applies just as well to an odd number $n$ where $\sigma(n) \equiv 2n $ (mod $4$). 

Since perfect numbers are rare, they give us little empirical data to work with. In the case of odd perfect numbers, proving statements about them is fraught with difficulty in that one lacks concrete examples to help check that one's claims or proofs are not nonsense or fallacious. A natural inclination is to try to generalize the notion of perfect numbers. In that context,  Sierpi\'nski \cite{WS} defined a number $n$ to be \textit{pseudoperfect} if there is a subset $S$ of the proper divisors of $n$ such that $\sum_{d \in S} d = n$. We refer to the set $S$ as a {\it pseudoperfect representation} of $n$; we also use this same term to refer to the equation which involves the sum of the divisors. Thus for example, we say that $\{2, 4, 6\}$ is a pseudoperfect representation of $12$, but can also say that $2+4+6=12$ is a pseudoperfect representation of $12$; in practice this slight abuse of language should not create confusion. 

 The pseudoperfect numbers are sequence \seqnum{A005835} in the OEIS, and the first few are  

$$6, 12, 18, 20, 24, 28, 30, 36, 40, \ldots .$$ 

Given a pseudoperfect representation of $n$, we refer to a number $d$ which is a divisor of $n$ and is not in $S$ as an \textit{omitted divisor}. Note that a number may be pseudoperfect in more than one way. For example, $12$ has both the representation $2+4+6=12$ already mentioned and the representation $1+2+3+6=12$. In the first example the omitted divisors are $1$ and $3$ while in the second example the only omitted divisor is $4$. 
Pseudoperfect numbers with a fixed small number of omitted divisors have been a subject of study over the last few years  \cite{AMPSZ, CCEKLM, EH, FFPSZZ, LL, PS, RC, TMF}. Pollack and Shevelev \cite{PS} defined a number to be $k$-near-perfect if the number is pseudoperfect with exactly $k$ omitted divisors, and defined a number to be just near-perfect if it is $1$-near-perfect. 

Pollack and Shevelev found three families of near-perfect numbers, \begin{enumerate}
\item $2^{t-1}(2^t-2^k-1)$ where $2^t-2^k-1$ is prime. Here $2^k$ is the omitted divisor.
\item $2^{2p-1}(2^p-1)$ where $2^p-1$ is prime. Here $2^p(2^p-1)$ is the omitted divisor.
\item $2^{p-1}(2^p-1)^2$ where $2^p-1$ is prime. Here $2^p-1$ is the omitted divisor.
\end{enumerate}

Subsequent work by Ren and Chen \cite{RC} showed that if $n$ is a near-perfect number with exactly 2 distinct prime divisors, then either $n=40$ or $n$ is one of the three families in question. Li and Liao \cite{LL} gave a partial classification of near-perfect numbers with three distinct prime factors.  Tang, Ma, and Feng \cite{TMF} showed that $173369889=3^4 \cdot 7^2 \cdot 11^2 \cdot 19^2 $ is the only odd near-perfect number with four or fewer distinct prime divisors. Whether there are infinitely many odd near-perfect numbers is open; it is not hard to show that every odd near-perfect number must be a perfect square.  Hasanalizade \cite{EH} studied near-perfect numbers which are also Fibonacci or Lucas numbers.

Aryan, Madhavani, Parikh,  Slattery, and the second author \cite{AMPSZ}  of this paper studied 2-near-perfect numbers. Their results included a classification of all 2-near-perfect numbers of the form $2^m p$ and $2^m p^2$. Subsequent work \cite{FFPSZZ} classified all 2-near-perfect numbers with exactly two distinct prime factors. 

Cohen, Cordwell, Epstein, Kwan, Lott and Miller \cite{CCEKLM} gave strong upper and lower bounds on the number of $k$-near-perfect numbers for fixed $k \geq 4$. 

While perfect numbers are rare, pseudoperfect numbers are very common. It is not hard to show that if $n$ is pseudoperfect, then so is $mn$. Benkoski and Erd\H{o}s \cite{BE} conjectured that every odd abundant number is pseudoperfect. It is known, due to unpublished work by Daniel Larsen \cite{Larsen}, that there is a constant $C$ such that if $\sigma(n) >Cn$, then $n$ is pseudoperfect.  Thus, pseudoperfect numbers have almost the opposite problem of perfect numbers: they are so common that there is not much one can say about all of them as a set.

A positive integer $n$ is said to be \textit{strongly pseudoperfect} if there is a subset $S$ of the positive divisors of $n$ such that $d \in S$ if and only if $n/d \in S$, and $\sum_{d \in S} d = 2n$.

 We can talk about a strongly pseudoperfect representation of $n$ in a way analogous to that for pseudoperfect numbers, and similarly talk about omitted divisors given that representation. 
Strongly pseudoperfect numbers fit in a middle ground: they are common but not as common as pseudoperfect numbers. 

The strongly pseudoperfect numbers form sequence \seqnum{A334405} with the definition first suggested by Amiram Eldar, and with the name ``strongly pseudoperfect'' introduced by the second author and McCormack \cite{MZ}.  The list of strongly pseudoperfect numbers begins

$$6, 28, 36, 60, 84, 90, 120, 156, 210, 216, 240, \ldots . $$

Note that if $n$ is a strongly pseudoperfect number with strongly pseudoperfect representation $S$, and $n \in S$, then $n$ is pseudoperfect with pseudoperfect representation $S - \{n\}$. However, if $n$ has strongly pseudoperfect representation $S$ and $n \not \in S$, it is not obvious that $n$ is pseudoperfect, since it may be the case that the sum of the elements of $S$ is $2n$ but no subset of the proper divisors of $n$ sums to $n$. However,  we are not aware of an example of such a phenomenon, and we suspect it does not happen. Motivated by this,  the second author and McCormack \cite{MZ} defined a number $n$ to be \textit{extremely strongly pseudoperfect} if $n$ is strongly pseudoperfect with representation $S$ and where $n \in S$. 

Just as a pseudoperfect number may have more than one pseudoperfect representation, strongly pseudoperfect numbers also can have more than one strongly pseudoperfect representation. For example, $156$ has a strongly pseudoperfect representation where $(2,78)$ is omitted but it also has a strongly pseudoperfect representation where the pairs $(3,52)$ and $(12,13)$ are omitted. 

Descartes noted that $D=198585576189$ looks almost like an odd perfect number. In  particular, one may factor it as $D = 3^2 \cdot 7^2 \cdot 11^2 \cdot 13^2 \cdot  22021$. One has then that 
$$\sigma(D) = (3^2+3+1)(7^2+7+1)(11^2+11+1)(13^2 +13+1)(22021+1)= 2D,$$
where we ignore that $22021 = 19^2 \cdot 61$  is in fact not a prime number.  

Voight \cite{Voight} noted the similar example, $V=3^4 \cdot 7^2 \cdot 11^2\cdot  19^2 (-127)^1$ where $$(3^4 +3^3+3^2 +3+1)(7^2 +7+1)(11^2+11+1)(19^2 +19+1)(-127+1)=2V,$$ and where the calculation is treating $-127$ as if it were a positive prime number. 

A large collaboration led by Pace Nielsen \cite{pacespoofgroup} generalized Descartes and Voight's observations as follows. Given an integer $n$, define a {\emph{factorization}} of $n$ to be an expression of the form  $$n = \prod_{i=1}^{k} x_i^{b_i},$$ where the $x_i$ are integers and the $b_i$ are positive integers. Notice that a factorization can also be thought of as a multiset of ordered pairs of the form $(x_i,b_i)$.  

Define a function $\tilde{\sigma}$ on the multiset of such ordered pairs by
$$
\tilde{\sigma}\Big(\{(x_i,b_i):1\leq i\leq k\}\Big) =\prod_{i=1}^{k}\left(\sum_{j=0}^{b_i}x_i^{j} \right).
$$

A factorization as above is \emph{spoof perfect} if $\tilde{\sigma}(\prod_{i=1}^{k}x_i^{b_i})=2\prod_{i=1}^{k}x_i^{b_i}$. Here the sum on the left is defined in terms of the formal product rather than the value of the product, or if one prefers, is defined on the multiset of $(x_i, b_i)$ pairs. If $n=p_1^{a_1}p_2^{a_2}\cdots p_k^{a_k}$ with the $p_i$ distinct primes, then $$\tilde{\sigma}(\{(p_1,a_1),(p_2,a_2), \ldots (p_k,a_k)\})= \sigma(n).$$ Thus $\tilde{\sigma}$ provides a generalization of the classical $\sigma$ function.  Every actual perfect number $n$ gives rise to a spoof perfect factorization. 

In this context, Descartes's number $D$ gives rise to a spoof factorization given by the ordered pairs $\{(3,2),(7,2),(11,2), (13,2),(22021,1)\}$. The number $D$ with this factorization is the only known example where all the $x_i$ are positive and odd. Many other examples exist where one is allowed to have negative values for $x_i$. If all the $x_i$ are greater than 1, and pairwise relatively prime, then the resulting expanded sum gives rise to an extremely strongly pseudoperfect representation. 

The second author of this paper along with McCormack \cite{MZ} proved the following two results.

\begin{proposition} If $n$ satisfies $n \equiv 2$ (mod $3$) or $n \equiv 3$ (mod $4$), then $n$ is not strongly pseudoperfect.  \label{2 mod 3 and 3 mod 4}
\end{proposition}

\begin{proposition} If $p$ is a prime where $p > \sigma(n)$, then $pn$ is not strongly pseudoperfect. 
\end{proposition}

They also showed that every number $n$ of the form  $2^{k-1}(2^k-1)$ with $k \geq 2$ is strongly pseudoperfect; the proof is to essentially ``pretend'' that $2^k-1$ is prime, and ignore all divisors of $n$ that would arise from non-trivial factors of $2^k-1$, and then the proof is identical to that of Euclid.  

One question in that paper was whether there are infinitely many odd strongly pseudoperfect numbers. We show that the answer to that question is yes.  This is a corollary of one of our two main results.

\begin{theorem} Let $n$ be a strongly pseudoperfect number. If there exists a strongly pseudoperfect representation of $n$ in which the divisor pair $(1, n)$ is not omitted, then $n^k$ is strongly pseudoperfect for every positive integer $k$.  \label{Power result}
\end{theorem}

We also prove:

\begin{theorem}\label{Mersenne result} Suppose that $n$ is a strongly pseudoperfect number of the form $n=2^kp$ where $p=2^m-1$ and $p$ is prime. Then $n$ is perfect. \end{theorem}

Theorem \ref{Mersenne result} can be considered a generalization in some sense of Euler's direction of the Euclid--Euler theorem, since it is saying that the only strongly pseudoperfect numbers which arise of this form are precisely those arising from Euclid's construction.  Section \ref{sec:constructions} discusses new constructions for strongly pseudoperfect numbers, including the proof of Theorem \ref{Power result}. Section \ref{sec:restriction} discusses new restrictions that strongly pseudoperfect numbers must satisfy. Section \ref{sec:questions} discusses open questions. Section \ref{sec:computations} includes tables of data for various types of strongly pseudoperfect numbers. 

\section{New constructions}
\label{sec:constructions}

In this section, we present new constructions of strongly pseudoperfect numbers. We first prove Theorem \ref{Power result}.

\begin{proof}[Proof of Theorem \ref{Power result}]

We prove the theorem with the assumption that $n$ is not a perfect square. The proof is essentially identical when $n$ is a perfect square.

Let $S$ be a subset of the divisors of $n$ that does not contain the divisor pair $(1, n)$, and suppose that $d_i\in S$ if and only if $\frac{n}{d_i}\in S$. Set $q_i = n/d_i$. Let $T$ be the sum of $S$, and assume that

\[T=\sum_{i=1}^m(d_i+q_i)=n-1.\]
Note that $n$ is then a strongly pseudoperfect number with representation $S \cup \{1,n\}$. Let $S_k$ be a subset of the divisors of $n^k$, obtained by multiplying the elements in $S$ by $(1, n, \dots ,n^{k-1})$. Since all elements $d_i, q_i\in S$ are divisors of $n$, every element $n^jd_i, n^jq_i\in S_k$ divides $n^jn=n^{j+1}$, and therefore divides $n^k$ where $0\leq j\leq k-1$. Note that $n^jd_i$ are distinct since $n^j < n^jd_i < n^{j+1}$ because $1< d_i < n$.

Note that $S_k$ obeys the property that when a divisor $n^jd_i\in S_k$, its pair $n^{k-j-1}q_i$ is also in $S_k$ since $n^jd_i\times n^{k-j-1}q_i=n^k$. Since the pair $(1, n)$ is not contained in $S$, multiplying the elements in $S$ by $(n^0,n,\dots,n^{k-1})$ can never produce $1=1\times n^0$ or $n^k=n\times n^{k-1}$, so $1\notin S_k$ and $n^k\notin S_k$. 
Let $T_k$ be the sum of $S_k$, hence

\[T_k=(1+n+\dots+n^{k-1})\times T.\]

Substituting $T=n-1$ we have,

\[T_k=\frac{1-n^k}{1-n}\times(n-1)=n^k-1.\]
Since $1\notin S_k$ and $n^k\notin S_k$, we can add the divisor pair $1+n^k$ to $T_k$ to get $n^k-1+1+n^k=2n^k$, and thus $n^k$ is strongly pseudoperfect. Thus, we now have a strongly pseudoperfect representation for $n^k$ which contains $(1,n^k)$, and the hypothesis holds allowing us to iterate the construction. 
\end{proof}

We then have a corollary of Theorem \ref{Power result}.

\begin{corollary} There are infinitely many odd strongly pseudoperfect numbers. \label{Infinitely many odd strongly pseudoperfects}
\end{corollary}
\begin{proof} This follows from observing that $n=11025$ is an odd strongly pseudoperfect number which satisfies the hypotheses of Theorem \ref{Power result}. In particular, if $n=11025$, then $\sigma(n)= 22971$, and so we obtain a strongly pseudoperfect representation by omitting $105$ (which pairs with itself), and omitting the pairs $(25, 441)$ as well as $(35, 315)$. \end{proof}

The following result arose from a conversation with Anthropic's Claude AI.

\begin{proposition}  If $n = 3 \cdot 2^a \cdot (2^{a+1}-3)$ where $a\geq 2$, and $2^{a+1}-3$ is prime, then $n$ is strongly pseudoperfect with a representation with exactly four omitted divisors: $2$, $n/2$, $6$, and $n/6$. \label{Claude's family}    
\end{proposition}
Checking the claim of Proposition \ref{Claude's family} is straightforward. After Claude pointed out this family of solutions, we generalized the result further. 

A family like the new family above uses essentially a set of small omitted divisors $d_1, d_2, \ldots, d_k$ such that $\sigma(n)$ is very close to $2n + n(1/d_1 + 1/d_2 + \cdots +  1/d_k)$. With $n= 2^am$ for some fixed $m$, when $a$ is large $h(n)$ is very close to some specific rational which can be written this way. For example, the above family is using that $h(n)$ is very close to $8/3$ and $8/3 = 2 + 1/2 + 1/6$. More generally if $b$ is a positive integer and we set $q= 2^{b+1}-1$, then we note that $$2\frac{q+1}{q} = 2 + \frac{1}{2^b} + \frac{1}{(2^{b+1}-1) 2^b},$$ which motivates the next result.

\begin{theorem} Let $b$ be a fixed integer, and set $q= 2^{b+1}-1$. Let $p = 2^{a+1} - (2^{b}+1)$ with $a > b \geq 1$. Then $n=pq 2^a$ is a strongly pseudoperfect number. If $p$ and $q$ are prime, then $n$ is strongly pseudoperfect with a representation with exactly four omitted divisors. \label{Generalized family of strongly pseudoperfect} 
\end{theorem}
\begin{proof} It is a straightforward calculation that if $p$ and $q$ are both prime, then $n$ is strongly pseudoperfect with omitted pairs of divisors 
 $(2^b,  (qp2^a)/(2^b) )$ and $((2^{b+1}-1)2^b , qp2^a/((2^{b+1}-1)2^b))$. If $p$ or $q$ is not prime, then we can remove all pairs of divisors that arise from non-trivial factors of $p$ and $q$, so $n$ is still strongly pseudoperfect, but with more omitted divisors, with all the divisors creating new distinct pairs, essentially via ``pretending'' that $p$ and $q$ are prime in the same way as the generalization of the Euclid--Euler construction.
\end{proof}

Note that $b=1$ in Theorem \ref{Generalized family of strongly pseudoperfect} is exactly the family in Proposition \ref{Claude's family}.

Theorem \ref{Power result} and Corollary \ref{Infinitely many odd strongly pseudoperfects} show that there are infinitely many odd strongly pseudoperfect numbers. Theorem \ref{Power result} motivates the following definition. A number $n$ is said to be \textit{primitive strongly pseudoperfect} if $n$ is strongly pseudoperfect and $n$ is not of the form $n=m^k$, where $m$ is a strongly pseudoperfect number and $k \geq 2$. This is an analog to the standard notion of a primitive pseudoperfect number, which is a number $n$ which is pseudoperfect and is not a multiple of a smaller pseudoperfect number. Simply using multiples in this definition does not work since a multiple of a strongly pseudoperfect number may end up not being strongly pseudoperfect. However, Theorem \ref{Power result} shows that powers behave in a way similar to how multiples behaved in the pseudoperfect situation. 

An earlier version of this paper asked two questions.

\begin{question} Are there infinitely many odd numbers which are primitive strongly pseudoperfect? \label{infinitely many primitive strongly pseudoperfect} \end{question}

Since no squarefree number greater than 1 is a non-trivial power, it is worth noting that $937365= 3 \cdot 5 \cdot 11 \cdot 13 \cdot 19 \cdot 23$ is squarefree and strongly pseudoperfect. If there are infinitely many such numbers then this would provide an alternate proof to answer Question \ref{infinitely many primitive strongly pseudoperfect}. 

Recall Theorem \ref{Euler's theorem for odd perfects}, that every odd perfect number $n$ must be of Euler's form, that is $n=p^e m^2$ where $p$ is an odd prime and $p \equiv e \equiv 1 $ (mod $4$), and $m$ is an odd integer such that $p \not \mid m$. If there were no strongly pseudoperfect numbers of Euler's form, then there would be no odd perfect numbers. However, we do not get so lucky. In particular, $1102725=3^2 \cdot 5^2 \cdot 13^2 \cdot 29$ is of Euler's form and is strongly pseudoperfect with a representation which includes $(1,n)$. If $n$ is of Euler's form, then so is $n^5$, so there are infinitely many strongly pseudoperfect numbers of Euler's form. We thus ask the following.

\begin{question} Are there infinitely many odd $n$ where $n$ is primitive strongly pseudoperfect, and $n=p^e m^2$ where $p$ is an odd prime and $p \equiv e \equiv  1$ (mod $4$), such that $p \not \mid m$? \label{Euler form question} \end{question}

After a draft of this paper was posted on the arXiv, an anonymous user on Reddit, TacoYaci, shared that they had found answers to both these questions via prompting an AI system.\cite{Redditconversation} Their AI approach was a ``contract system between Codex (ChatGPT Astra 6 Ultra) and Claude Code (Fable 5.1 max).'' With their permission, we include here a variant of their constructions which answers both these questions in the affirmative.

Their answers to Question \ref{infinitely many primitive strongly pseudoperfect} and Question \ref{Euler form question} both rest on the following result.

\begin{theorem} If $a$ and $b$ are extremely strongly pseudoperfect numbers then $n= a^2b$ is an extremely strongly pseudoperfect number. \label{TacoYaci-a-squared-b}
\end{theorem}
\begin{proof} The proof is essentially a variation of the proof of Theorem \ref{Power result}. Given an extremely strongly pseudoperfect number, $n$, we write $T_n$ for a set satisfying three conditions:

\begin{enumerate}
    \item The set $T_n$ is a subset of the divisors of $n$, and $\{1,n\} \cap   T_n = \emptyset$.
    \item For all $d$, $d \in T_n$ if and only if $n/d \in T_n$.
    \item We have the relationship $$\sum_{d \in T_n} d = n-1.$$
\end{enumerate}
Notice that a number $n$ having a set $T_n$ satisfying these three conditions is exactly equivalent to $n$ being extremely strongly pseudoperfect. Given a set of numbers $S$, we write $kS = \{ks \mid s \in S\}$.

Now, assume that $a$ and $b$ are both extremely strongly pseudoperfect. Thus, we have a set $T_a$ and a set $T_b$ each of which satisfies the three conditions above. Now consider the set $T_n= T_a \cup aT_b \cup abT_a $.  If we can show that this set $T_n$ satisfies all three conditions above for $n=a^2b$,  then we will have proven the result. Note that every element of $T_n$ is a divisor of $n$. 

First note that $T_a$, $aT_b$ and $abT_a$ are all disjoint sets.  To see this, note that $T_a$ is a subset of the interval $(1,a)$, $aT_b$ is a subset of  the interval $(a,ab)$, and $abT_a$ is a subset of the interval $(ab,a^2b)$, and none of these intervals overlap. (We can in fact, narrow these intervals slightly more, but we do not need to do so here.) We also note that $d \in T_a$ if and only if $n/d \in abT_a$, since $\frac{a^2b}{d} = ab\frac{a}{d} \in abT_a$. Also, if $d \in aT_b$, then $n/d \in aT_b$, since if $d \in aT_b$, then $d= ad_1$ where $d_1 \in T_b$, and so $$n/d = \frac{a^2b}{d} = \frac{a^2b}{a d_1} = a \frac{b}{d_1} \in aT_b,$$ where the last step is using that since $d_1 \in T_b$, we have that $b/d_1 \in T_b$. Finally, we note that the sum of the elements of $T_n$ is $n-1$ since 

$$\sum_{d \in T_n} d = \sum_{d \in T_a} d + \sum_{d \in a T_b} d + \sum_{d \in abT_a } d = (a-1) + a(b-1) + ab(a-1) = a^2b-1=n-1.$$
\end{proof}

To answer Question \ref{infinitely many primitive strongly pseudoperfect} and Question \ref{Euler form question}, we set $x=11025=3^2 \cdot 5^2 \cdot 7^2$,  $y=  972405 = 3^4 \cdot  5 \cdot 7^4$. It is not hard to check that each of $x$ and $y$ is extremely strongly pseudoperfect. For each we list the omitted divisors in pairs below, except for $105$ for $x$ which is its own pair as the square root of 11025:

\begin{align*}
O_x&=\{25,\ 441,\ 35,\ 315,\ 105\},\\ 
O_y &= \{35,\ 27783,\ 45,\ 21609,\ 49,\ 19845,\ 105,\ 9261,\ 135,\ 7203,\ 441,\ 2205\}.
 \end{align*}

Then by Theorem \ref{Power result} and Theorem \ref{TacoYaci-a-squared-b}, for all positive $k$,  every number of the form $x^{2k} y$ is an extremely strongly pseudoperfect number of Euler's form. Moreover, if $k \not \equiv 2$ (mod $3$), then $x^{2k}y$ is an odd primitive strongly pseudoperfect number. 

\section{New restrictions}
\label{sec:restriction}
\begin{proposition} Assume that $n=3 \cdot 2^a p$ where $p$ is a prime at least $5$ and $a \geq 1$. Assume further that $n$ is strongly pseudoperfect with exactly one pair of divisors omitted. Then $n =156 $ and the omitted divisors are $(2,78)$. \label{3 2 to a times p, exactly one pair of divisors omitted}
\end{proposition}

\begin{proof}  Assume that $n=3 \cdot 2^ap$ where $p$ is a prime at least 5, and $a \geq 1$. Assume further that $n$ is strongly pseudoperfect with exactly one pair of divisors omitted. We have $\sigma(n)=4(2^{a+1}-1)(p+1)$. Assume the omitted pair of divisors is $x$ and $y$. So $xy= 3\cdot 2^ap $, and without loss of generality we may assume that $p \mid y$ and $p \nmid x$.  We have then

\begin{equation} 4(2^{a+1}-1)(p+1) - x -y = 2n=3 \cdot 2^{a+1}p. \label{equation which knocks out a single pair of divisors} \end{equation}

We have two cases to consider: First, $x=2^b$ for some $b \leq a$. Second, $x=3 \cdot 2^b$ for some $b \leq a$.

Case I: We have $x=2^b$ and so $y= 3\cdot 2^{a-b}p$. Then from Equation \eqref{equation which knocks out a single pair of divisors} we have

\begin{equation} 4(2^{a+1}-1)(p+1) - 2^b - 3 \cdot 2^{a-b}p = 3 \cdot 2^{a+1} p. \label{Case I x, y} \end{equation}
Set  $2^a=A$, $2^b=B$, and solve for $p$ to obtain:

\begin{equation} p= \frac{B(B+4 - 8A)}{A(2B-3)-4B}. \label{Case I, solved for p} \end{equation}

 The numerator of Equation \eqref{Case I, solved for p} is negative, and so the denominator must also be negative. Thus, we need that $A(2B-3)-4B<0$
and so $4B> A(2B-3)$. But since $A \geq B$, and $A$ and $B$ both are powers of $2$, this can only happen if $B=1$ or $B=2$.

Assume that $B=1$. Thus, from Equation \eqref{Case I, solved for p} we obtain that $p= \frac{5-8A}{-A-4} = (8A-5)/(A+4)$. But then the numerator  is odd (since $A$ is a power of 2 greater than 1) and denominator is even, so there are no integer solutions.

Now, assume that $B=2$. We then have from Equation \eqref{Case I, solved for p} that  $p= \frac{2(6-8A)}{A-8}= \frac{2(8A-6)}{8-A}$. Since $B=2$ and  $4B > A(2B-3)$, we have that $A < 8$. Since $B \leq A< 8$ and both $A$ and $B$ are powers of 2, either $A=2$ or $A=4$. If $A=2$, $p$ is not an integer. If $A=4$, then $p= 13$, and $n=156$, and $x=2$ and $y=78$. 

Case II: Assume that $x= 3 \cdot 2^b$. 

We get the equation:
\begin{equation} 4(2^{a+1}-1)(p+1) - 3 \cdot 2^b - 2^{a-b}p = 3 \cdot 2^{a+1}p. \label{Case II x, y} \end{equation}

We again set $2^a=A$, $2^b=B$, and solve for $p$ to obtain that

$$p = \frac{B(3B+4 - 8A)}{A(2B-1)-4B}.$$

But $3B+4 - 8A$ is negative, and so $A(2B-1)-4B$ must also be negative in order for $p$ to be positive. Thus, $A(2B-1) <4B $. If $B=1$, this becomes $A < 4$, and so the only options then are $A=1$ and $A=2$, neither of which gives an integer value for $p$.  If $B =2$, then $3A < 8$, and so the only option is $A=B=2$,  which yields $p=6$ which is not prime.  If $B \geq 4$, then $A(2B-1) \geq 7A > 4B$, and so we have a contradiction.
\end{proof}

\begin{lemma}  Let $n$ be a strongly pseudoperfect number of the form $n= 3 \cdot 2^a p$ for some prime $p >3$. Assume further that $a \geq 5$ and $n$ has exactly four omitted divisors. Then one of the omitted divisors is $1$ or $2$. \label{Four omitted divisors, one is 1 or 2} Moreover, if the smallest omitted divisor is $2$, then the second smallest omitted divisor is one of $3$, $4$, $5$, $6$, $7$, or $8$.
\end{lemma}
\begin{proof} From $a \geq 5$, we have $\sigma(n) > \frac{21}{8}n$. Let the omitted divisor pairs be $(x_1, x_2)$ and $(y_1, y_2)$ with $(x_1 < x_2)$ and $(y_1 < y_2)$. Note that $x_1 < \sqrt{n}$ and $y_1 < \sqrt{n}$. Assume further that $x_1 < y_1$. If $x_1 \geq 3$, then $x_2 \leq \frac{n}{3}$, and $y_2 \leq \frac{n}{4}$. We thus have

$$2n=\sigma(n) - x_1 - x_2 - y_1 - y_2 \geq \frac{21}{8}n - \frac{n}{3} - \frac{n}{4} - 2\sqrt{n},$$
or equivalently $\frac{1}{24}n \leq 2\sqrt{n}, $ which forces $n \leq 2304$. A quick check then confirms that $n=84$ and $n = 156$ are the only numbers of the desired form under this bound which do not have $x_1=1$ or $x_1=2$. Since both have $a < 5$, we do not need to worry about them. 

Thus, we must have $x_1=1$ or $x_1 =2$.

Now, assume that the smallest omitted divisor is $2$, and assume that the next omitted divisor is greater than $8$. Since $9$ cannot be a divisor of $n$, we have then that the next omitted divisor must be at least $10$. Then we have

$$2n= \sigma(n) - 2 - \frac{n}{2} - y_1 - \frac{n}{y_1} \geq   \frac{21}{8}n -2 - \frac{n}{2} - \sqrt{n} - \frac{n}{10},$$
which forces $n \leq 1756$. But this is a finite check and we can see that the only such $n$ is $n=552$, which again does not have $a \geq 5$.

\end{proof}

\begin{proposition}If $n$ is a strongly pseudoperfect number of the form $n= 3 \cdot 2^a p$ for some prime $p >3$ and $a \geq 1$, and $n$ has exactly four omitted divisors, and where $a \geq 5$, then $p=2^{a+1}-3$ and the omitted divisors are exactly $(2, n/2)$ and $(6, n/6)$.  \label{no strongly pseudoperfect of form 3 times p times 2 to a and a at least 5, and exactly two pairs of strong}
\end{proposition}
\begin{proof}  Based on Lemma \ref{Four omitted divisors, one is 1 or 2}, we have the smallest omitted divisor is either $1$ or $2$.
Case I: $x_1=1$
We have then 

\begin{equation} 2n = \sigma(n) - 1 - n - y_1 - \frac{n}{y_1} \end{equation}
for some divisor $y_1$, or equivalently 

\begin{equation} 3p2^{a+1} = 4(p+1)(2^{a+1}-1) - 1 - 3p2^{a}- y_1- \frac{3p2^{a}}{y_1}.  \label{Case I, x1 is 1 explicit form}\end{equation}

Considering Equation \eqref{Case I, x1 is 1 explicit form} modulo 2 we get that $y_1+ \frac{3p2^{a}}{y_1}$ is odd. Since the only odd divisors of $n$ are $1,3, p,$ and $3p$, and $1$ is already an omitted divisor, we may without loss of generality in this case assume that $y_1$ is the odd divisor. We thus have only three cases to consider $y_1 =3$, $y_1=p$ and $y_1=3p$.

In the case $y_1=3$, Equation \eqref{Case I, x1 is 1 explicit form} becomes

\begin{equation} 3p2^{a+1} = 4(p+1)(2^{a+1}-1) - 4 - 4p2^{a} \label{Case I, x1=1, y1=3}.
    \end{equation}
But then the left-hand side of Equation \eqref{Case I, x1=1, y1=3} is divisible by 8 while the right-hand side is congruent to $4$ (mod $8$), and so we have no solution.

In the next situation, we have $y_1=p$. In this case, Equation \eqref{Case I, x1 is 1 explicit form} becomes:

\begin{equation} 3p2^{a+1} = 4(p+1)(2^{a+1}-1) - 1 - 3p2^{a}- p- 3 \cdot 2^a \label{Case I, x1 =1, y1=p}.
\end{equation}
If we set $A=2^a$, and solve for $p$ Equation \eqref{Case I, x1 =1, y1=p} becomes

\begin{equation}
p=\frac{5A-5}{A+5},    
\end{equation}
and so we must have $A+5 \mid 5A-5$, which means that $A+5 \mid 5(A+5)-(5A-5)=30$, but $A$ is at least 32 so this is impossible. 

Our last case is $y_1=3p$. Then Equation \eqref{Case I, x1 is 1 explicit form} becomes

\begin{equation} 3p2^{a+1} = 4(p+1)(2^{a+1}-1) - 1 - 3p2^{a}- 3p- 2^a.  \label{Case I, x1 =1, y2=3p}\end{equation}

We set $A=2^a$, and Equation \eqref{Case I, x1 =1, y2=3p} becomes

$$p=\frac{7A-5}{A+7},$$
    and so $A+7 \mid 7A-5$, and $A+7 \mid 7(A+7) - (7A-5)= 54$. This yields a contradiction since $A \geq 32$, and thus $A+7 \geq 39$ and the only divisor of $54$ which is greater than $39$ is $54$, and $54-7=47$ is not a power of 2. 

Case II: We assume that $x_1=2$. We obtain then

\begin{equation}
    2n = \sigma(n) - 2 - \frac{n}{2} - y_1 - \frac{n}{y_1}, 
\end{equation} 
or 
\begin{equation}
   3p2^{a+1} = 4(p+1)(2^{a+1}-1) - 2 - 3p2^{a-1} - y_1 - \frac{3p2^a}{y_1}.
\end{equation}

In this case, by Lemma \ref{Four omitted divisors, one is 1 or 2}, we have $y_1 \in \{3,4,5,6, 7,8\}$. However, a parity argument shows that we must have $y_1$ even, and so we have $y_1=4$, $y_1=6$, or $y_1=8$.

First, consider the case where $y_1=4$. In this situation we have

\begin{equation} 3p2^{a+1} = 4(p+1)(2^{a+1}-1) - 6 - 3p2^{a-1} -  3p2^{a-2}. \label{Case II, x1=2, y1=4}
\end{equation}
If we set $2^a=A$, and solve for $p$ in Equation \eqref{Case II, x1=2, y1=4},we get

\begin{equation}
     6pA = 4(p+1)(2A-1) - 6 - (3/2)pA -  (3/4)pA,
\end{equation}
and so
$$p = \frac{32A-40}{A+16},$$
which via similar logic to earlier yields that $A+16 \mid 32(A+16)-(32A-40)=552$ which has no solutions.

Next, we consider the case where $y_1=6$. We have then

\begin{equation}
    3p2^{a+1} = 4(p+1)(2^{a+1}-1) - 8 - 3p2^{a-1} - p2^{a-1}.
\end{equation}
We as before set $2^a=A$, and solve for $p$ to obtain $p= 2A-3$. 

Finally, we consider the case where $y_1=8$. We have then

    \begin{equation}
3p2^{a+1} = 4(p+1)(2^{a+1}-1) - 10 - 3p2^{a-1}-  \frac{3p2^{a}}{8}.\label{almost end equation}
        \end{equation}
Taking Equation (\ref{almost end equation}) and solving for $p$  we obtain that 
$$    p=\frac{16(4A-7)}{32-A},$$
which is negative for $a>5$, and thus has no solutions. The case of $a=5$ would involve a division by zero and so also does not work. \end{proof}

\begin{proposition} Let $n$ be a strongly pseudoperfect number of the form $n= 3 \cdot 2^a p$ for some prime $p >3 $, where $n$ has exactly four omitted divisors. Then either $n \in \{84, 156,552\}$, or $n = 3 \cdot 2^a (2^{a+1}-3)$ with omitted divisor pairs $(2,n/2)$ and $(6,n/6)$. \label{3 times power of 2 times p and four omitted}
\end{proposition}
\begin{proof} By Proposition \ref{no strongly pseudoperfect of form 3 times p times 2 to a and a at least 5, and exactly two pairs of strong}, we may assume that $a < 5$.

We have then for some divisors $x_1$ and $y_1$,  where we may assume that $\gcd(x_1y_1,p)=1$, that

\begin{equation} 2n=\sigma(n) - x_1 -\frac{n}{x_1} - y_1 - \frac{n}{y_1}. \label{2n = sigma with two omitted pairs again} \end{equation}

Note that we are using a different assumption about $x_1$ and $y_1$ here. While in the previous proof, we chose $x_1$ and $y_1$ as the smallest omitted divisors, here we are choosing them to be the pair relatively prime to $p$. 

Equation \eqref{2n = sigma with two omitted pairs again} becomes

\begin{equation}
    3p2^{a+1} = 4(p+1)(2^{a+1}-1) - x_1 - \frac{3p2^a}{x_1} - y_1 - \frac{3p2^a}{y_1},
\end{equation}
which after we clear denominators becomes

\begin{equation}
    3p2^{a+1}x_1 y_1 =  4(p+1)(2^{a+1}-1)x_1y_1 - x_1^2y_1 - 3p2^ay_1 - y_1^2x_1 - 3p2^a x_1,
\end{equation}
    and thus we have $$p \mid 4(p+1)(2^{a+1}-1)x_1y_1 -   x_1^2y_1 - y_1^2x_1=x_1y_1(4(p+1)(2^{a+1}-1) -x_1 -y_1),$$ and since $\gcd(x_1y_1,p)=1$, this becomes

$$    p \mid 4(p+1)(2^{a+1}-1) -   x_1 - y_1,  $$
     which implies that

\begin{equation}
    p \mid 4(2^{a+1}-1) - x_1 -y_1.\label{2nd p divisibility argument}
\end{equation}

We note that $4(2^{a+1}-1) - x_1 -y_1 >0$, since $x_1, y_1 \leq 3 \cdot 2^a$, and $x_1 \neq y_1$ and so $p < 4(2^{a+1}-1)$. Since $a \leq 4$, we have 
    $p \leq 124$, and since $p$ is prime we in fact have $p \leq 113$. Thus, we have $n \leq 113 \cdot 16 \cdot 3= 5424$. This is now a finite set of cases to check.  We get three values which are not in the family, $n=84$, $n=156$, and $n=552$. 

\end{proof}

We now prove Theorem \ref{Mersenne result}.

\begin{proof}[Proof of Theorem \ref{Mersenne result}]
    
Let $n$ be a strongly pseudoperfect number of the form $n=2^kp$ where $p=2^m-1$ and $p$ is prime. Then

\begin{equation}
    \sigma(n)=(1+p)(2^{k+1}-1)=2^m(2^{k+1}-1).
\end{equation}

Let $A$ be the sum of omitted divisors. Then

\begin{equation}
    A=\sigma(n)-2n=2^m(2^{k+1}-1)-2^{k+1}(2^m-1)=2^{k+1}-2^m.
\end{equation}

Since $n$ is strongly pseudoperfect, $A\geq0$. Thus $2^{k+1}\geq2^m$ and $k+1\geq m$. We therefore have two cases: Case 1 where $k+1=m$ and Case 2 where $k+1>m$.

\textbf{Case 1}

Assume that $k+1=m$. Then

\begin{equation}
    A=2^{k+1}-2^m=0,
\end{equation}
so $\sigma(n)=2n$ and $n$ is perfect.

\textbf{Case 2}

Assume that $k+1>m$. Then we have

\begin{equation}
    A=2^m(2^{k+1-m}-1),
\end{equation}
and so $2^m\mid A$. The divisors of $n$ occur in pairs of the form $(2^i,2^{k-i}(2^m-1))$. Define $s_i$ as the sum of each omitted pair, where

\begin{equation}
    s_i=2^i+2^{k-i}(2^m-1)
\end{equation}
for $0\leq i\leq k$. We have two cases: Case 2A where an omitted pair has a divisor with a small power of 2 satisfying $i<m$ or $k-i<m$, and Case 2B where the omitted pairs are within the middle divisors of $n$ satisfying $i\geq m$ and $k-i\geq m$.

\textbf{Case 2A:} $i<m$ or $k-i<m$

Assume that at least one omitted divisor pair of $n$ satisfies $i<m$ or $k-i<m$. Among the divisor pairs with this property, choose the pair containing a divisor with the smallest possible power of 2. Let this smallest exponent be $v$, where $v<m$, and suppose no omitted pair has a divisor with a power of 2 less than $v$.

Since $v$ is the smallest possible exponent of 2 appearing in each omitted divisor pair, every omitted pair containing $2^v$ must have one of two forms: $(2^v,2^{k-v}(2^m-1))$ or $(2^{k-v},2^v(2^m-1))$.

If neither of these pairs were omitted, this would contradict the choice of $v$.

Every divisor of $n$ is either a power of 2 or a power of 2 times $2^m-1$. Since $2^m-1$ is odd, the exact power of 2 dividing a divisor is determined by the exponent of 2 that appears explicitly. The only divisors divisible by $2^v$ but not $2^{v+1}$ are $2^v$ and $2^v(2^m-1)$, appearing in the pairs above. All other omitted pairs are divisible by both $2^v$ and $2^{v+1}$ since their smallest exponent is larger than $v$.

If exactly one of the two pairs above is omitted, the total sum of omitted pairs contains an unpaired $2^v$ term and is therefore not divisible by $2^{v+1}$. Since $v+1\leq m$ and $2^m\mid A$, this is impossible. Therefore, if one of the two pairs is omitted, the other must also be omitted.
Their combined sum is

\begin{equation}
    (2^v+2^{k-v})2^m.
\end{equation}

Since $v<m$, we have $k-v+m\geq k+1$, so their sum is at least $2^{k+1}$ and therefore exceeds $A=2^{k+1}-2^m$. Thus both pairs cannot be included in the omitted sum.
In the case where $v=k-v$, the only pair containing $2^v$ is $(2^v,2^v(2^m-1))$, and its sum is

\begin{equation}
    s_v=2^v+2^v(2^m-1)=2^{v+m}.
\end{equation}

Since $k=2v$ and $v<m$, we have $m\geq v+1$. Therefore

\begin{equation}
    s_v=2^{v+m}\geq2^{2v+1}=2^{k+1}>A,
\end{equation}
so this pair is also too large to be omitted. Therefore it is enough to rule out the smallest exponent that appears. Hence all omitted pairs must satisfy $i\geq m$ and $k-i\geq m$.

Thus we must rely on Case 2B to find middle omitted factors that sum to $A$.

\textbf{Case 2B:} $i\geq m$ and $k-i\geq m$

Since all omitted pairs must satisfy $i\geq m$ and $k-i\geq m$, or equivalently $m\leq i\leq k-m$, the largest possible sum of omitted divisors is obtained by taking all such pairs.  Then

\begin{equation}
    T = \sum_{i=m}^{k-m}\left(2^i+2^{k-i}(2^m-1)\right)
\end{equation}
is an upper bound for this sum.

Since $\sum_{i=m}^{k-m}2^{k-i}=\sum_{i=m}^{k-m}2^i$, we get

\begin{equation}
    T=2^m\sum_{i=m}^{k-m}2^i
    =2^m(2^{k-m+1}-2^m)
    =2^{k+1}-2^{2m}.
\end{equation}

But

\begin{equation}
    2^{k+1}-2^{2m}<2^{k+1}-2^m=A,
\end{equation}
so even the sum of all omitted divisors with $m\leq i\leq k-m$ is too small. Note that if $k - m < m$, then the above sums are empty, but then $T=0<A$, and so the conclusion still holds.

Case 2A proved that a divisor pair containing a divisor whose exact power of $2$ has an exponent less than $m$ cannot be part of the omitted factors. Case 2B proved that the maximum sum of all remaining divisors is too low to reach the required omitted sum $A$. Therefore Case 2 is impossible. The only valid case is Case 1 where $k+1=m$, and therefore $n$ is perfect.
\end{proof}

\begin{proposition} Fix a positive integer $m >2$. Assume $n$ is a positive integer such that all prime factors of $n$ are $\pm 1 $ (mod ${m}$) and that an odd number of prime factors (counted with multiplicity) $p$ satisfy $p \equiv -1 $ (mod $m$). Then $n$ is not strongly pseudoperfect. \label{1 and -1 mod m result}
\end{proposition}
\begin{proof}  Note that $n \equiv -1$ (mod $m$). For each divisor $d$ of $n$, exactly one of $d$ and $n/d$ is $1$ (mod $m$) and the other is $-1 $ (mod  $m$), so $d + \frac{n}{d} \equiv 0 $ (mod $m$). Thus, if $S$ is a strongly pseudoperfect representation of $n$, then $\sum_{d \in S} d  \equiv 0 $ (mod $m$). But $2n \equiv -2$ (mod $m$), and since $m>2$, this is a contradiction, since we cannot have $0 \equiv -2$ (mod $m$).
\end{proof}

Proposition \ref{1 and -1 mod m result} is a generalization of Proposition \ref{2 mod 3 and 3 mod 4} since the only way a number $n$ can be $2$ (mod $3$) is if it has all prime factors congruent to $\pm 1 $ (mod $3$) and with an odd number of them congruent to $-1$ (mod $3$), and a similar remark applies to a number being $3$ (mod $4$).

\begin{proposition} If $n$ is not a perfect square, and $0 < \sigma(n) -2n < 2\sqrt{n}$, then $n$ is not strongly pseudoperfect.   
\end{proposition}
\begin{proof} Assume $n$ satisfies $0 < \sigma(n) -2n < 2\sqrt{n}$. Since $n$ is not perfect, every strongly pseudoperfect representation $S$ of $n$ must have at least one omitted pair. But the function $f(x)= x + n/x$ has a minimum of $2\sqrt{n}$ and so the sum of the elements of $S$ is at most $\sigma(n) - 2\sqrt{n}$ which is less than $2n$.
\end{proof}

\begin{proposition} Suppose that $n$ is not a perfect square, is strongly pseudoperfect and all prime divisors of $n$ are $1 $ (mod $3$). Suppose further that $\tau(n) \equiv 0 $ (mod $3$). Then $\sigma(n) -2n > 4\sqrt{n}$.
\end{proposition}
\begin{proof} Note that we have $\sigma(n) \equiv 0$ (mod $3$). Since $\sigma(n)$ and $2n$ disagree modulo 3, $n$ is not perfect.  For every divisor $d$ of $n$, $d + \frac{n}{d} \equiv 2$ (mod $3$).

Now note that for every divisor $d$, $\sigma(n) - d - \frac{n}{d} \equiv 0 - (2) \equiv 1 $ (mod $3$). But $2n \equiv 2 $ (mod $3$), and so $2n$ cannot be equal to $\sigma(n) - d - \frac{n}{d}$, and thus every strongly pseudoperfect representation of $n$ must omit at least two divisor pairs, and so $\sigma(n) > 2n + 4\sqrt{n}$.
\end{proof}

We can generalize this slightly to obtain:

\begin{theorem}
 Let $m$ be an odd number. Suppose that $n$ is not a perfect square, is strongly pseudoperfect and all prime divisors of $n$ are $1 $ (mod $m$) for $m >1$. Suppose further that $\tau(n) \equiv 0$ (mod $m$). Then $\sigma(n) -2n > 2(m-1)\sqrt{n}$.   
\end{theorem}
\begin{proof} Since all divisors of $n$ are $1$ (mod $m$), we have $\sigma(n) \equiv \tau(n) \equiv 0$ (mod $m$), and so omitting $k$ pairs yields $2k \equiv -2$ (mod $m$), and thus $k \equiv -1 $ (mod $m$), and so $k \geq m-1$ from which the result follows, since each pair of divisors has a sum of at least $2\sqrt{n}$.
    \end{proof}

\section{Open questions and further work}
\label{sec:questions}

 Given Larsen's result, an obvious question is whether there is some constant $C_2$ and some set of reasonably simple conditions where those conditions along with $\sigma(n)>C_2n$, imply that $n$ is strongly pseudoperfect. At first glance, Proposition \ref{1 and -1 mod m result} and Theorem \ref{Mersenne result} suggest that this is not the case. But Theorem \ref{Mersenne result} is less of an obstruction than one might think since all the numbers that satisfy its hypotheses have $\sigma(n)/n < 8/3$.

Theorem \ref{Power result} required that $n$ be a strongly pseudoperfect number with a representation which contains $(1,n)$. Thus, if every strongly pseudoperfect number is extremely strongly pseudoperfect then we would have the nice result that if $n$ is strongly pseudoperfect, so is $n^k$. Thus, we also have the following related question.

\begin{question} For every $n$ where $n$ is a strongly pseudoperfect number, is $n^k$ strongly pseudoperfect for all $k$?
\end{question}

Another closely related question is the following.

\begin{question} Does the set of strongly pseudoperfect numbers have positive natural density?
\end{question}

Finally, it is well known that if $m$ is a perfect or abundant number, then $mn$ is abundant for all $n>1$. Motivated by this, a number $n$ is said to be primitive non-deficient if $n$ is not deficient and if no proper divisor of $n$ is non-deficient.

\begin{question} Are there infinitely many numbers which are odd, primitive non-deficient and strongly pseudoperfect? \label{odd primitive strong pseudo q}
\end{question}

\section{Computations and data}

This section includes tables for odd strongly pseudoperfect numbers, even strongly pseudoperfect numbers, and odd primitive non-deficient strongly pseudoperfect numbers. 

\label{sec:computations}

\begin{center}
\breakablemathcommas
\small
\begin{longtable}{|c|c|p{11cm}|}
\hline
$n$ & \textbf{\#} & \textbf{Omitted divisors} \\
\hline
\endfirsthead
\hline
$n$ & \textbf{\#} & \textbf{Omitted divisors} \\
\hline
\endhead
\multicolumn{3}{r}{\textit{Continued on next page}}
\endfoot
\caption{\normalsize The first ten odd strongly pseudoperfect numbers and their omitted divisors. The \# column gives the number of solutions. For entries with many solutions, only the first few are displayed.}\label{tab:odd-solutions}
\endlastfoot
 
11025 & 1 &
$\{25,441,35,315,105\}$ \\*
\hline\noalign{\nobreak}
 
12285 & 3 &
$\{15,819,27,455,35,351,39,315,65,189\}$ \newline
$\{15,819,21,585,35,351,63,195,91,135\}$ \newline
$\{13,945,27,455,35,351,63,195,91,135\}$ \\
\hline
 
16065 & 6 &
$\{15,1071,45,357,51,315,63,255,105,153\}$ \newline
$\{15,1071,35,459,63,255,85,189,105,153\}$ \newline
$\{21,765,27,595,45,357,51,315,119,135\}$ \newline
$\{17,945,35,459,45,357,63,255,119,135\}$ \newline
$\{21,765,27,595,35,459,85,189,119,135\}$ \newline
$\{17,945,27,595,63,255,85,189,119,135\}$ \\
\hline
 
17325 & 1 &
$\{9,1925,21,825,33,525,45,385,105,165\}$ \\
\hline
 
23625 & 1 &
$\{21,1125,45,525,105,225,125,189,135,175\}$ \\
\hline
 
28665 & 3 &
$\{15,1911,21,1365,35,819,91,315,147,195\}$ \newline
$\{21,1365,35,819,39,735,49,585,63,455,91,315,147,195\}$ \newline
$\{15,1911,35,819,63,455,65,441,91,315,117,245,147,195\}$ \\
\hline
 
31185 & 29 &
$\{11,2835,15,2079,33,945,35,891,77,405\}$ \newline
$\{11,2835,15,2079,27,1155,45,693,81,385\}$ \newline
$\{11,2835,15,2079,21,1485,81,385,99,315\}$ \newline
$\{15,2079,21,1485,27,1155,35,891,45,693,81,385,99,315\}$ \newline
$\{7,4455,33,945,35,891,63,495,105,297\}$ \newline
$\{11,2835,27,1155,33,945,45,693,55,567,63,495,105,297\}$ \newline
$\{11,2835,21,1485,35,891,55,567,63,495,81,385,105,297\}$ \newline
$\{11,2835,21,1485,33,945,55,567,63,495,99,315, 105,297\}$ \newline
$\{21,1485,27,1155,33,945,35,891,45,693,55,567,63,495,99,315,105,297\}$ \newline
$\{11,2835,15,2079,55,567,77,405,81,385,99,315,105,297\}$ \newline
$\vdots$ \newline
\textit{(29 solutions total)} \\
\hline
 
36225 & 5 &
$\{15,2415,45,805,63,575,75,483,105,345\}$ \newline
$\{21,1725,23,1575,69,525,75,483,115,315\}$ \newline
$\{21,1725,23,1575,63,575,75,483,161,225\}$ \newline
$\{15,2415,35,1035,69,525,105,345,175,207\}$ \newline
$\{25,1449,45,805,63,575,69,525,75,483,115,315,175,207\}$ \\
\hline
 
42525 & 1 &
$\{27,1575,35,1215,45,945,63,675,75,567\}$ \\
\hline
 
45045 & 695 &
$\{5,9009,21,2145,35,1287,39,1155,45,1001\}$ \newline
$\{9,5005,15,3003,21,2145,33,1365,35,1287,45,1001,63,715\}$ \newline
$\{7,6435,13,3465,15,3003,45,1001,65,693\}$ \newline
$\{9,5005,11,4095,21,2145,45,1001,55,819,63,715,65,693\}$ \newline
$\{9,5005,11,4095,33,1365,35,1287,39,1155,45,1001,77,585\}$ \newline
$\{13,3465,15,3003,21,2145,33,1365,35,1287,45,1001,55,819,63,715,77,585\}$ \newline
$\{11,4095,13,3465,15,3003,33,1365,35,1287,65,693,77,585\}$ \newline
$\{11,4095,13,3465,33,1365,35,1287,39,1155,45,1001,63,715,65,693,77,585\}$ \newline
$\{9,5005,15,3003,33,1365,39,1155,45,1001,55,819,63,715,65,693,77,585\}$ \newline
$\{9,5005,13,3465,21,2145,33,1365,35,1287,63,715,91,495\}$ \newline
$\vdots$ \newline
\textit{(695 solutions total)} \\
\hline
 
\end{longtable}
\end{center}

\begin{center}
\breakablemathcommas
\small
\begin{longtable}{|c|c|p{11cm}|}
\hline
$n$ & \textbf{\#} & \textbf{Omitted divisors} \\
\hline
\endfirsthead
\hline
$n$ & \textbf{\#} & \textbf{Omitted divisors} \\
\hline
\endhead
\multicolumn{3}{r}{\textit{Continued on next page}}
\endfoot
\caption{\normalsize The first twenty even strongly pseudoperfect numbers and their omitted divisors.}\label{tab:even-solutions}
\endlastfoot
6 & 1 & Perfect \\*
\hline
28 & 1 & Perfect \\
\hline
36 & 1 & $\{4,9,6\}$ \\
\hline
60 & 1 & $\{2,30,6,10\}$ \\
\hline
84 & 1 & $\{3,28,4,21\}$ \\
\hline
90 & 1 & $\{3,30,6,15\}$ \\
\hline
120 & 1 & $\{3,40,5,24,6,20,10,12\}$ \\
\hline
 
156 & 2 &
$\{2,78\}$ \newline
$\{3,52,12,13\}$ \\
\hline
 
210 & 1 &
$\{5,42,6,35,7,30,10,21\}$ \\
\hline
 
216 & 3 &
$\{2,108,4,54\}$ \newline
$\{3,72,4,54,8,27\}$ \newline
$\{4,54,6,36,8,27,9,24\}$ \\
\hline
 
240 & 3 &
$\{2,120,4,60,6,40,12,20\}$ \newline
$\{4,60,5,48,6,40,8,30,12,20,15,16\}$ \newline
$\{3,80,6,40,8,30,10,24,12,20,15,16\}$ \\
\hline
 
252 & 1 &
$\{3,84,4,63,9,28,12,21\}$ \\
\hline
 
270 & 1 &
$\{5,54,6,45,10,27,15,18\}$ \\
\hline
 
300 & 1 &
$\{2,150,4,75,12,25\}$ \\
\hline
 
312 & 1 &
$\{2,156,6,52\}$ \\
\hline
 
330 & 1 &
$\{2,165,15,22\}$ \\
\hline
 
336 & 5 &
$\{2,168,4,84,6,56\}$ \newline
$\{3,112,4,84,6,56,7,48\}$ \newline
$\{2,168,6,56,8,42,14,24\}$ \newline
$\{3,112,6,56,7,48,8,42,14,24\}$ \newline
$\{4,84,6,56,7,48,12,28,14,24,16,21\}$ \\
\hline
 
352 & 1 &
$\{8,44\}$ \\
\hline
 
396 & 2 &
$\{2,198,9,44,11,36\}$ \newline
$\{3,132,6,66,9,44,18,22\}$ \\
\hline
 
420 & 4 &
$\{3,140,4,105,5,84,7,60,10,42,14,30\}$ \newline
$\{3,140,4,105,5,84,6,70,14,30,15,28\}$ \newline
$\{2,210,5,84,7,60,10,42,15,28,20,21\}$ \newline
$\{2,210,4,105,10,42,12,35,15,28,20,21\}$ \\
\hline
 
\end{longtable}
\end{center}
 
\begin{center}
\breakablemathcommas
\small
\begin{longtable}{|c|c|p{11cm}|}
\hline
$n$ & \textbf{\#} & \textbf{Omitted divisors} \\
\hline
\endfirsthead
\hline
$n$ & \textbf{\#} & \textbf{Omitted divisors} \\
\hline
\endhead
\multicolumn{3}{r}{\textit{Continued on next page}}
\endfoot
\caption{\normalsize The first ten odd primitive non-deficient, strongly pseudoperfect numbers and their omitted divisors.}\label{tab:odd-primitive-solutions}
\endlastfoot
 
153153 & 2 &
$\{51,3003,63,2431,221,693,231,663,273,561\}$ \newline
$\{91,1683,119,1287,143,1071,153,1001,221,693,231,663,273,561\}$ \\*
\hline\noalign{\nobreak}
 
182325 & 6 &
$\{55,3315,65,2805,85,2145,255,715,325,561\}$ \newline
$\{75,2431,85,2145,143,1275,165,1105,221,825,255,715,325,561\}$ \newline
$\{65,2805,85,2145,165,1105,187,975,255,715,275,663,325,561\}$ \newline
$\{65,2805,85,2145,143,1275,221,825,255,715,275,663,425,429\}$ \newline
$\{75,2431,85,2145,143,1275,165,1105,187,975,325,561,425,429\}$ \newline
$\{65,2805,85,2145,143,1275,195,935,275,663,325,561,425,429\}$ \\
\hline
 
218025 & 3 &
$\{45,4845,57,3825,153,1425\}$ \newline
$\{45,4845,135,1615,153,1425,225,969,425,513\}$ \newline
$\{51,4275,75,2907,255,855,323,675,459,475\}$ \\
\hline
 
255645 & 1 &
$\{95,2691,115,2223,117,2185,171,1495,207,1235,285,897,299,855\}$ \\
\hline
 
263925 & 2 &
$\{75,3519,85,3105,425,621\}$ \newline
$\{115,2295,153,1725,225,1173,345,765,459,575\}$ \\
\hline
 
272745 & 2 &
$\{33,8265,95,2871,145,1881,209,1305,261,1045\}$ \newline
$\{45,6061,99,2755,145,1881,209,1305,261,1045,285,957,435,627\}$ \\
\hline
 
309225 & 4 &
$\{57,5425,93,3325,95,3255,105,2945,465,665\}$ \newline
$\{35,8835,93,3325,217,1425,285,1085,465,665\}$ \newline
$\{57,5425,93,3325,133,2325,217,1425,399,775,465,665,475,651\}$ \newline
$\{75,4123,93,3325,105,2945,155,1995,285,1085,465,665,525,589\}$ \\
\hline
 
330165 & 10 &
$\{69,4785,115,2871,145,2277,207,1595,261,1265\}$ \newline
$\{45,7337,69,4785,319,1035\}$ \newline
$\{69,4785,99,3335,145,2277,261,1265,319,1035\}$ \newline
$\{115,2871,145,2277,165,2001,207,1595,253,1305,319,1035,345,957\}$ \newline
$\{87,3795,165,2001,207,1595,253,1305,261,1265,319,1035,345,957\}$ \newline
$\{69,4785,207,1595,253,1305,261,1265,319,1035,345,957,435,759\}$ \newline
$\{69,4785,115,2871,145,2277,165,2001,495,667\}$ \newline
$\{69,4785,87,3795,165,2001,261,1265,495,667\}$ \newline
$\{45,7337,165,2001,261,1265,319,1035,495,667\}$ \newline
$\{69,4785,165,2001,253,1305,319,1035,345,957,435,759,495,667\}$ \\
\hline
 
344565 & 1 &
$\{117,2945,155,2223,247,1395,285,1209,585,589\}$ \\
\hline
 
347985 & 3 &
$\{33,10545,99,3515,495,703\}$ \newline
$\{57,6105,95,3663,171,2035,185,1881,495,703\}$ \newline
$\{57,6105,171,2035,209,1665,285,1221,407,855,495,703,555,627\}$ \\
\hline
 
\end{longtable}
\end{center}

Code and further data can be found at the GitHub repository:\\
\href{https://github.com/audreywang0/Strongly-Pseudoperfect}{https://github.com/audreywang0/Strongly-Pseudoperfect}

\section{Acknowledgments}
\label{sec:Acknowledgements}
Pace Nielsen provided helpful input and caught many small issues with an early version of the paper. 

\textbf{Leiden Declaration Disclosure:} The family from Proposition \ref{Claude's family} was discovered by Anthropic's Claude when the second author gave it an incorrect version of Proposition \ref{3 times power of 2 times p and four omitted}. Claude noticed an error in the argument, and pointed out that the result was false due to the family in question. Theorem \ref{TacoYaci-a-squared-b} was pointed out by Redditor TacoYaci who found it by applying an AI to an early preprint version of this paper, and supplied then the construction involving the extremely strongly pseudoperfect numbers $x$ and $y$. No other use of AI occurred in the research for this paper.

\bigskip
\hrule
\bigskip

\noindent 2020 {\it Mathematics Subject Classification}:
Primary 11A25; Secondary 11N64, 11A05, 11N25

\noindent \emph{Keywords: }  abundant number, perfect number, primitive non-deficient number,  pseudoperfect number, strongly pseudoperfect number,  spoof perfect number.

\bigskip
\hrule
\bigskip

\noindent (Concerned with sequences
\seqnum{A000396},  \seqnum{A005835}, \seqnum{A006039}, \seqnum{A334405}.)

\end{document}